\documentclass[12pt,reqno]{amsart}
\usepackage{amscd,amssymb,amsmath,tikz}
\begin{document}

\newtheorem{thm}[equation]{Theorem}
\numberwithin{equation}{section}
\newtheorem{cor}[equation]{Corollary}
\newtheorem{expl}[equation]{Example}
\newtheorem{rmk}[equation]{Remark}
\newtheorem{conv}[equation]{Convention}
\newtheorem{claim}[equation]{Claim}
\newtheorem{lem}[equation]{Lemma}
\newtheorem{sublem}[equation]{Sublemma}
\newtheorem{conj}[equation]{Conjecture}
\newtheorem{defin}[equation]{Definition}
\newtheorem{diag}[equation]{Diagram}
\newtheorem{prop}[equation]{Proposition}
\newtheorem{notation}[equation]{Notation}
\newtheorem{tab}[equation]{Table}
\newtheorem{fig}[equation]{Figure}
\newcounter{bean}
\renewcommand{\theequation}{\thesection.\arabic{equation}}

\raggedbottom \voffset=-.7truein \hoffset=0truein \vsize=8truein
\hsize=6truein \textheight=8truein \textwidth=6truein
\baselineskip=18truept

\def\mapright#1{\ \smash{\mathop{\longrightarrow}\limits^{#1}}\ }
\def\mapleft#1{\smash{\mathop{\longleftarrow}\limits^{#1}}}
\def\mapup#1{\Big\uparrow\rlap{$\vcenter {\hbox {$#1$}}$}}
\def\mapdown#1{\Big\downarrow\rlap{$\vcenter {\hbox {$\ssize{#1}$}}$}}
\def\mapne#1{\nearrow\rlap{$\vcenter {\hbox {$#1$}}$}}
\def\mapse#1{\searrow\rlap{$\vcenter {\hbox {$\ssize{#1}$}}$}}
\def\mapr#1{\smash{\mathop{\rightarrow}\limits^{#1}}}
\def\ss{\smallskip}
\def\s{\sigma}
\def\l{\lambda}
\def\vp{v_1^{-1}\pi}
\def\at{{\widetilde\alpha}}

\def\sm{\wedge}
\def\la{\langle}
\def\ra{\rangle}
\def\ev{\text{ev}}
\def\od{\text{od}}
\def\on{\operatorname}
\def\ol#1{\overline{#1}{}}
\def\spin{\on{Spin}}
\def\cat{\on{cat}}
\def\lbar{\ell}
\def\qed{\quad\rule{8pt}{8pt}\bigskip}
\def\ssize{\scriptstyle}
\def\a{\alpha}
\def\bz{{\Bbb Z}}
\def\Rhat{\hat{R}}
\def\im{\on{im}}
\def\ct{\widetilde{C}}
\def\ext{\on{Ext}}
\def\sq{\on{Sq}}
\def\eps{\epsilon}
\def\ar#1{\stackrel {#1}{\rightarrow}}
\def\br{{\bold R}}
\def\bC{{\bold C}}
\def\bA{{\bold A}}
\def\bB{{\bold B}}
\def\bD{{\bold D}}
\def\bC{{\bold C}}
\def\bh{{\bold H}}
\def\bQ{{\bold Q}}
\def\bP{{\bold P}}
\def\bx{{\bold x}}
\def\bo{{\bold{bo}}}
\def\dh{\widehat{d}}
\def\si{\sigma}
\def\Vbar{{\overline V}}
\def\dbar{{\overline d}}
\def\wbar{{\overline w}}
\def\Sum{\sum}
\def\tfrac{\textstyle\frac}

\def\tb{\textstyle\binom}
\def\Si{\Sigma}
\def\w{\wedge}
\def\equ{\begin{equation}}
\def\b{\beta}
\def\G{\Gamma}
\def\L{\Lambda}
\def\g{\gamma}
\def\d{\delta}
\def\k{\kappa}
\def\psit{\widetilde{\Psi}}
\def\tht{\widetilde{\Theta}}
\def\psiu{{\underline{\Psi}}}
\def\thu{{\underline{\Theta}}}
\def\aee{A_{\text{ee}}}
\def\aeo{A_{\text{eo}}}
\def\aoo{A_{\text{oo}}}
\def\aoe{A_{\text{oe}}}
\def\vbar{{\overline v}}
\def\endeq{\end{equation}}
\def\sn{S^{2n+1}}
\def\zp{\bold Z_p}
\def\cR{{\mathcal R}}
\def\P{{\mathcal P}}
\def\cQ{{\mathcal Q}}
\def\cj{{\cal J}}
\def\zt{{\bold Z}_2}
\def\bs{{\bold s}}
\def\bof{{\bold f}}
\def\bq{{\bold Q}}
\def\be{{\bold e}}
\def\Hom{\on{Hom}}
\def\ker{\on{ker}}
\def\kot{\widetilde{KO}}
\def\coker{\on{coker}}
\def\da{\downarrow}
\def\colim{\operatornamewithlimits{colim}}
\def\zphat{\bz_2^\wedge}
\def\io{\iota}
\def\om{\omega}
\def\Prod{\prod}
\def\e{{\cal E}}
\def\zlt{\Z_{(2)}}
\def\exp{\on{exp}}
\def\abar{{\overline a}}
\def\xbar{{\overline x}}
\def\ybar{{\overline y}}
\def\zbar{{\overline z}}
\def\mbar{{\overline m}}
\def\nbar{{\overline n}}
\def\sbar{{\overline s}}
\def\kbar{{\overline k}}
\def\bbar{{\overline b}}
\def\et{{\widetilde E}}
\def\ni{\noindent}
\def\tsum{\textstyle \sum}
\def\coef{\on{coef}}
\def\den{\on{den}}
\def\lcm{\on{l.c.m.}}
\def\vi{v_1^{-1}}
\def\ot{\otimes}
\def\psibar{{\overline\psi}}
\def\thbar{{\overline\theta}}
\def\mhat{{\hat m}}
\def\exc{\on{exc}}
\def\ms{\medskip}
\def\ehat{{\hat e}}
\def\etao{{\eta_{\text{od}}}}
\def\etae{{\eta_{\text{ev}}}}
\def\dirlim{\operatornamewithlimits{dirlim}}
\def\gt{\widetilde{L}}
\def\lt{\widetilde{\lambda}}
\def\st{\widetilde{s}}
\def\ft{\widetilde{f}}
\def\sgd{\on{sgd}}
\def\lfl{\lfloor}
\def\rfl{\rfloor}
\def\ord{\on{ord}}
\def\gd{{\on{gd}}}
\def\rk{{{\on{rk}}_2}}
\def\nbar{{\overline{n}}}
\def\MC{\on{MC}}
\def\lg{{\on{lg}}}
\def\cH{\mathcal{H}}
\def\cS{\mathcal{S}}
\def\cP{\mathcal{P}}
\def\N{{\Bbb N}}
\def\Z{{\Bbb Z}}
\def\Q{{\Bbb Q}}
\def\R{{\Bbb R}}
\def\C{{\Bbb C}}

\def\mo{\on{mod}}
\def\xt{\times}
\def\notimm{\not\subseteq}
\def\Remark{\noindent{\it  Remark}}
\def\kut{\widetilde{KU}}

\def\*#1{\mathbf{#1}}
\def\0{$\*0$}
\def\1{$\*1$}
\def\22{$(\*2,\*2)$}
\def\33{$(\*3,\*3)$}
\def\ss{\smallskip}
\def\ssum{\sum\limits}
\def\dsum{\displaystyle\sum}
\def\la{\langle}
\def\ra{\rangle}
\def\on{\operatorname}
\def\proj{\on{proj}}
\def\od{\text{od}}
\def\ev{\text{ev}}
\def\o{\on{o}}
\def\U{\on{U}}
\def\lg{\on{lg}}
\def\a{\alpha}
\def\bz{{\Bbb Z}}
\def\eps{\varepsilon}
\def\bc{{\bold C}}
\def\bN{{\bold N}}
\def\bB{{\bold B}}
\def\bW{{\bold W}}
\def\nut{\widetilde{\nu}}
\def\tfrac{\textstyle\frac}
\def\b{\beta}
\def\G{\Gamma}
\def\g{\gamma}
\def\zt{{\Bbb Z}_2}
\def\zth{{\bold Z}_2^\wedge}
\def\bs{{\bold s}}
\def\bx{{\bold x}}
\def\bof{{\bold f}}
\def\bq{{\bold Q}}
\def\be{{\bold e}}
\def\lline{\rule{.6in}{.6pt}}
\def\xb{{\overline x}}
\def\xbar{{\overline x}}
\def\ybar{{\overline y}}
\def\zbar{{\overline z}}
\def\ebar{{\overline \be}}
\def\nbar{{\overline n}}
\def\ubar{{\overline u}}
\def\bbar{{\overline b}}
\def\et{{\widetilde e}}
\def\lf{\lfloor}
\def\rf{\rfloor}
\def\ni{\noindent}
\def\ms{\medskip}
\def\Dhat{{\widehat D}}
\def\what{{\widehat w}}
\def\Yhat{{\widehat Y}}
\def\abar{{\overline{a}}}
\def\minp{\min\nolimits'}
\def\sb{{$\ssize\bullet$}}
\def\mul{\on{mul}}
\def\N{{\Bbb N}}
\def\Z{{\Bbb Z}}
\def\S{\Sigma}
\def\Q{{\Bbb Q}}
\def\R{{\Bbb R}}
\def\C{{\Bbb C}}
\def\Xb{\overline{X}}
\def\eb{\overline{e}}
\def\notint{\cancel\cap}
\def\cS{\mathcal S}
\def\cR{\mathcal R}
\def\el{\ell}
\def\TC{\on{TC}}
\def\GC{\on{GC}}
\def\wgt{\on{wgt}}
\def\wpt{\widetilde{p_2}}
\def\wbar{\overline w}
\def\dstyle{\displaystyle}
\def\Sq{\on{sq}}
\def\Om{\Omega}
\def\ds{\dstyle}
\def\tz{tikzpicture}
\def\zcl{\on{zcl}}
\def\bd{\bold{d}}
\def\io{\iota}
\def\Vb#1{{\overline{V_{#1}}}}

\title
{Stiefel-Whitney classes and immersions of orientable and Spin manifolds}
\author{Donald M. Davis}
\address{Department of Mathematics, Lehigh University\\Bethlehem, PA 18015, USA}
\email{dmd1@lehigh.edu}
\author{W. Stephen Wilson}
\address{Department of Mathematics, Johns Hopkins University\\Baltimore, MD 01220, USA}
\email{wwilson3@jhu.edu}
\date{March 21, 2021}


\maketitle

\begin{abstract} We determine a nice simple formula for the largest Euclidean space for which there is an orientable $n$-manifold with a nonimmersion detected by Stiefel-Whitney classes. For Spin manifolds, we prove the analogue of the upper bound and  establish the complete answer for $n\le23$ and $32\le n\le33$. Results similar to many of these were obtained some 50 years ago, but in a much less tractable form. The sharp results for Spin manifolds require detailed calculations of $ko$-homology groups of mod-2 Eilenberg MacLane spaces.
\end{abstract}

\section{Introduction}\label{intro}
This work was motivated by a question asked by Mike Hopkins after Ralph Cohen's talk (\cite{Cohtalk}) on immersions of manifolds at a distinguished Harvard lecture series. Cohen had discussed aspects of his proof (\cite{Coh}) that every $n$-manifold can be immersed in $\R^{2n-\a(n)}$, where $\a(n)$ denotes the number of 1's in the binary expansion of $n$. Hopkins asked whether there were similar results for other classes of manifolds, such as orientable or Spin  manifolds.
Work was done on this question long ago for orientable manifolds in \cite{MP}, \cite{BP}, and  \cite{P}, and for Spin manifolds in \cite{R} and \cite{W}. We  extend their results and reinterpret  in a much more tractable form, with a self-contained proof.

By ``manifold'' we always mean a compact connected smooth manifold without boundary.
Let $\wbar_i$ denote the $i$th Stiefel-Whitney class of the stable normal bundle of a manifold. A standard result says that if an $n$-manifold $M$ immerses in $\R^{n+c}$, then $\wbar_i(M)=0$ for $i>c$. We say that a nonimmersion of an $n$-manifold in $\R^{n+c}$ is detected by Stiefel-Whitney classes if $\wbar_{i}(M)\ne0$ for some $i>c$.

Our user-friendly reinterpretation of \cite[Theorem 1]{P}
 is as follows.
\begin{thm}\label{or} Let $\eps_n=0$ if $n\equiv1$ mod 4, and otherwise $\eps_n=1$. There exists a nonimmersion of an orientable $n$-manifold in $\R^{2n-k-1}$ detected by Stiefel-Whitney classes if and only if $k\ge\a(n)+\eps_n$.\end{thm}

Thus for $n\equiv1$ mod 4, the restriction of Cohen's result to orientable manifolds is optimal, while for $n\not\equiv1$ mod 4, the best that one might hope for is that all orientable $n$-manifolds can be immersed in $\R^{2n-\a(n)-1}$.

The situation for Spin manifolds is similar, but more complicated, and is not completely resolved. The reduction of the problem to algebraic topology for both orientable and Spin manifolds is given in the following result, whose proof appears at the end of this section. Here $\chi$ is the canonical antiautomorphism of the mod 2 Steenrod algebra, $\iota_k$ is the fundamental class in the mod-2 cohomology of the Eilenberg MacLane space $K(\zt,k)$, and $ko_*(-)$ is connective $KO$ homology, localized at 2.
\begin{thm}\label{genlthm} \begin{itemize}

\mbox{}
\item[a.] Let $\rho:H_*(X;\Z)\to H_*(X;\zt)$ be induced by reduction mod $2$. There exists  an orientable  $n$-dimensional manifold with a nonimmersion in $\R^{2n-k-1}$ implied by Stiefel-Whitney classes if and only if there exists an element $\a\in H_n(K(\zt,k);\Z)$ such that $\la\chi\sq^{n-k}\io_k,\rho(\a)\ra\ne0$. Moreover, it is necessary that $\chi\sq^{n-k}\io_k\not\in\im(\sq^1)$.
\item[b.] Let $h:ko_*(X)\to H_*(X;\zt)$ denote the Hurewicz homomorphism. There exists  an $n$-dimensional Spin manifold with a nonimmersion in $\R^{2n-k-1}$ implied by Stiefel-Whitney classes if and only if there exists an element $\a\in ko_n(K(\zt,k))$ such that $\la\chi\sq^{n-k}\io_k,h_*\a\ra\ne0$. Moreover, it is necessary that $\chi\sq^{n-k}\io_k\not\in\im(\sq^1,\sq^2)$.
    \end{itemize}\end{thm}

In Section \ref{pfsec}, we prove the following theorem, which resolves completely the necessary conditions of Theorem \ref{genlthm}.

\begin{minipage}{6.5in}

\begin{thm}\label{chithm}

\begin{itemize}
\mbox{}

\item[i.] The smallest $k$ such that $\chi\sq^{n-k}\io_k\not\in\im(\sq^1)\subset H^n(K(\zt,k);\zt)$ is
$$\begin{cases}({\bf a})\ \a(m)+b&n=4m+b,\ 1\le b\le3\\
({\bf b})\ \a(n)+1&n\equiv0\pmod4.\end{cases}$$
\item[ii.] The smallest $k$ such that $\chi\sq^{n-k}\io_k\not\in\im(\sq^1,\sq^2)\subset H^n(K(\zt,k);\zt)$ is
$$\begin{cases}({\bf c})\ \a(m)+b&n=8m+b,\ 1\le b\le7\\
({\bf d})\ \a(n)+1&n\equiv2^e\pmod{2^{e+2}}, \ e\ge3\\
({\bf e})\ \a(n)+2&n\equiv3\cdot2^e\pmod{2^{e+2}},\ e\ge3.\end{cases}$$
\end{itemize}
\end{thm}
\end{minipage}

\bigskip

Immediate corollaries of Theorems \ref{genlthm} and \ref{chithm} are the ``only if'' part of Theorem \ref{or} and the following result.  One easily checks the equivalence of the ``$\a(m)+b$'' and ``$\a(n)+\eps'$'' versions.

\begin{cor}\label{spin} Define $\eps'_n$ by
$$\eps'_n=\begin{cases}0&n\equiv1\ (8)\\
1&n\equiv2,3\ (8)\\
3&n\equiv4,5\ (8)\\
4&n\equiv6,7\ (8)\\
1&n\equiv2^e \pmod{2^{e+2}},\ e\ge3\\
2&n\equiv3\cdot2^e\pmod{2^{e+2}},\ e\ge3.\end{cases}$$
If there exists an $n$-dimensional Spin manifold for which a nonimmersion in $\R^{2n-k-1}$ is detected by Stiefel-Whitney classes, then $k\ge\a(n)+\eps'_n$.
\end{cor}

It can be verified that Corollary \ref{spin} is equivalent to the less tractable result \cite[Proposition 1.1]{R}. However, part (ii) of Theorem \ref{chithm}, which is needed in the proof of Theorem \ref{spinnonimm}, is new.

The thing that makes the orientable case easier than the Spin case is that, as we show in Section \ref{ifsec}, for the minimal value of $k$ in case (i) of Theorem \ref{chithm}, a mod-2 homology class dual to $\chi\sq^{n-k}\io_k$ is always in the image from $H_n(K(\zt,k);\Z)$, thus implying the ``if'' part of Theorem \ref{or}.  In the Spin case, if $n$ is not one of the integers included in Theorem \ref{spinnonimm}, we have not yet been able to determine whether, for the minimal value of $k$ in case (ii) of Theorem \ref{chithm},
a mod-2 homology class dual to $\chi\sq^{n-k}\io_k$ is in the image from $ko_n(K(\zt,k))$. Moreover, for $n\in\{9,10,11,12,17,33\}$, we find that there is not a mod-2 homology class dual to $\chi\sq^{n-k}\io_k$ for the minimal possible value of $k$  in the image from $ko_n(K(\zt,k))$, but if we increase $k$ by 1, the appropriate class is in this image.
As we will discuss in Section \ref{ifsec}, many of these results were obtained, from a somewhat different perspective, by the second author in \cite{W}.
Our result is as follows.
\begin{thm}\label{spinnonimm} The largest value of $c$ for which there is an $n$-dimensional Spin manifold with $\wbar_c\ne0$ is given in Table \ref{T4}.
\begin{table}[h]
\caption{Nonzero dual Stiefel-Whitney classes}
\label{T4}

\begin{tabular}{c|ccccc}
n&$8$--$12$&$13$--$15$&$16$--$17$&$18$--$23$&$32$--$33$\\
\hline
c&$6$&$7$&$14$&$15$&$30$
\end{tabular}
\end{table}

All dual Stiefel-Whitney classes are $0$ in Spin manifolds of dimension less than $8$.
\end{thm}

Thus, for the values of $c$ in Theorem \ref{spinnonimm}, there exists an $n$-dimensional Spin manifold which does not immerse in $\R^{n+c-1}$, but Stiefel-Whitney classes allow the possibility that all immerse in $\R^{n+c}$.
For values of $n$ not included in Theorem \ref{spinnonimm}, we do not yet know the largest possible value of $c$.

We close this introductory section with this delayed proof.

\begin{proof}[Proof of Theorem \ref{genlthm}] We prove (b); the proof of (a) is similar, using \cite{CF}. We first prove the necessary condition.

Assume a nonimmersion of an $n$-manifold $M$ in $\R^{2n-k-1}$ is detected by $\wbar_{n-k}\ne0$. Then, by Poincar\'e
duality, there must exist a class $x\in H^{k}(M;\zt)$ such that $\wbar_{n-k} x$ is the nonzero element of $H^n(M;\zt)$. For a Spin manifold, the nonzero element of $H^n(M;\zt)$ is not in $\im(\sq^1,\sq^2)$. It is well-known (e.g., \cite{MP}) that $\wbar_{n-k} x=\chi\sq^{n-k}(x)$. Consideration of the map $f:X\to K(\zt,k)$ for which $f^*(\iota_{k})=x$ shows that $\chi\sq^{n-k}(\iota_{k})$ is not in $\im(\sq^1,\sq^2)$.

  The group $MSpin_n(X)=\pi_n(MSpin\w X)$ consists of cobordism classes of pairs $(M,f)$ where $M$ is an $n$-dimensional Spin manifold and $f:M\to X$ is a map. The Hurewicz homomorphism $MSpin_n(X)\to H_n(X;\zt)$ satisfies $h_*([M,f])=f_*(\rho([M]))$, where $[M]\in H_n(M;\Z)$ is the orientation class.
By \cite{ABP}, localized at 2, there is an equivalence $MSpin \to bo\vee W'$, where $W'$ is a 7-connected spectrum. Let $H\zt$ denote the mod-2 Eilenberg MacLane spectrum. The morphism $[MSpin,H\zt]\to[bo,H\zt]$ is an isomorphism, since $[W',H\zt]=0$.

There exists a nonimmersion of an $n$-dimensional Spin-manifold in $\R^{2n-k-1}$  detected by Stiefel-Whitney classes iff there is an $n$-dimensional Spin manifold $M$ and an element  $x\in H^k(M;\zt)$ such that $\la\chi\sq^{n-k}x,\rho[M]\ra\ne0$ iff there is an $n$-dimensional Spin manifold $M$ and a map $f:M\to K(\zt,k)$ such that $\la\chi\sq^{n-k}\io_k,f_*(\rho[M])\ra\ne0$ iff $\exists\a\in MSpin_n(K(\zt,k))$ such that
$\la\chi\sq^{n-k}\io_k,h_*\a\ra\ne0$ iff $\exists\a\in ko_n(K(\zt,k))$ such that
$\la\chi\sq^{n-k}\io_k,h_*\a\ra\ne0$.\end{proof}

\section{Proof of Theorem \ref{chithm}}\label{pfsec}

We use Milnor basis  and the following facts, where $\sq(R)=\sq(r_1,\ldots,r_s)$. (\cite{K}, \cite{Mil}) We assume that the reader is familiar with the complicated multiplication rule for Milnor basis elements.
\begin{prop}\label{Mil}
\begin{itemize}
\mbox{}
\item[i.] $|\sq(R)|=\sum(2^j-1)a_j$ and $\exc(R)=\sum a_j$.
\item[ii.] $\chi\sq^d$ is the sum of all $\sq(R)$ with $|\sq(R)|=d$.
\item[iii.] $\sq(R)\not\in\im(\sq^1,\sq^2)$ iff $r_1\equiv0$ mod 4 and $r_2\equiv0$ mod 2.
\item[iv.] $H^*(K(\zt,k);\zt)$ is a polynomial algebra generated by all $\sq(R)\iota_k$ for which $\exc(R)<k$.
\item[v.] $\sq(R)\iota_k=0$ if $\exc(R)>k$.
\item[vi.] If $R=(r_1,\ldots)$, $\exc(R)=k$, and $r_i=0$ for $i<t$, then $\sq(R)\iota_k=(\sq(S)\iota_k)^{2^t}$, where $S=(r_{t+1},\ldots)$.
\end{itemize}
\end{prop}

\begin{proof}[Proof of parts (a) and (c) of Theorem \ref{chithm}]
We prove part (c). The proof of part (a) is completely analogous.

Write $8m=\sum_{j\ge1}\eps_j2^j$ with $\eps_j\in\{0,1\}$. Then $\sq(E)=\sq(\eps_1,\ldots,\eps_r)$ has $|\sq(E)|=8m-\a(m)$, $\exc(E)=\a(m)$, and is not in $\im(\sq^1,\sq^2)$,  since $\eps_1=\eps_2=0$. With $k=\a(m)+b$, hence $n-k=8m-\a(m)$, $\chi\sq^{n-k}\iota_k$ contains the term $\sq(E)\iota_k$. This is part of the basis, since $\exc(E)<k$, can't be cancelled by other terms in $\chi\sq^{n-k}\iota_k$, and is not in $\im(\sq^1,\sq^2)$.

Now suppose $\sq(R)=\sq(r_1,\ldots,r_s)$ has $|\sq(R)|=n-\ell$ with $\ell\le \a(m)+b$, $\exc(R)\le\ell$, and $r_1\equiv0\mod 4$ and $r_2\equiv0$ mod 2. Then $\sum 2^jr_j=|\sq(R)|+\exc(R)\le n=8m+b$ implies \begin{equation}\label{deg}\sum 2^jr_j\le\sum2^j\eps_j=8m\end{equation} since $\sum 2^jr_j$ is a multiple of 8.  Let $b_j=r_j-\eps_j\ge-1$, and $r_1=4c_1$ and $r_2=2c_2$.  Then (\ref{deg}) implies
\begin{equation}\label{e1}8c_1+8c_2+\sum_{j\ge3}2^jb_j\le0,\end{equation}
while $\ell\le \a(m)+b$ implies $8m-\a(m)\le|\sq(R)|$ hence
\begin{equation}0\le 4c_1+6c_2+\sum_{j\ge3}(2^j-1)b_j.\label{e2}\end{equation}

We claim that the only solution of (\ref{e1}) and (\ref{e2}) with $c_j\ge0$ and $b_j\ge-1$ is the zero solution, which implies our result, namely that the only solution in part (b) with $k\le\a(m)+b$ is the one described at the beginning of the proof. First note that if there is a solution with $c_1$ or $c_2$ nonzero, they can be incorporated into $b_3$, so we may omit $c_1$ and $c_2$. Let $S=\{j:b_j=-1\}$. We wish to show that for a multiset of $t$'s (distinct from $S$ but not necessarily from one another), the only way to have $\sum 2^t\le\sum_S2^j$ and $\sum_S(2^j-1)\le\sum(2^t-1)$ is the empty sums. For example, having $b_j=2$ contributes two $2^t$'s with $t=j$.

Combining two equal $t$-terms makes the second inequality harder to satisfy. We perform this combining, and cancel whenever equal exponents occur on both sides. Thus we may assume all exponents are distinct. The largest exponent, $j$, must occur in $S$, and there is no way that distinct $(2^t-1)$'s less than that can be as large as $2^j-1$.\end{proof}

\medskip
\begin{proof}[Proof of part (b)] Let $n=4m=\sum_{j\ge1} 2^j\eps_j$ with $\eps_j\in\{0,1\}$ and $e\ge2$ the smallest subscript $j$ for which $\eps_j=1$. Note that $\a(m)=\a(n)$.

Suppose $R=(r_1,\ldots,r_s)$ has $|\sq(R)|=4m-\ell$ with $\ell\le\a(m)$, $\exc(R)\le\ell$, and $r_1\equiv0$ mod 2 (so $\sq(R)\not\in\im(\sq^1)$). Similarly to the proof of part (b), the only possibility is $r_j=\eps_j$ for all $j$. [\![$\sum 2^jr_j=|\sq(R)|+\exc(R)\le 4m=\sum 2^j\eps_j$. With $b_j=r_j-\eps_j\ge-1$ and $r_1=2c_1$, we get $4c_1+\sum_{j\ge2} 2^jb_j\le0$ and, from $4m-\a(m)\le |\sq(R)|$, $0\le2c_1+\sum_{j\ge2}(2^j-1)b_j$. As before, this has only the zero solution.]\!]

However,
\begin{eqnarray}\sq(\eps_1,\ldots,\eps_r)\iota_{\a(m)}&=&(\sq(\eps_e,\ldots,\eps_r)\iota_{\a(m)})^{2^{e-1}}\label{s1}\\
&=&\sq^1(\sq(0,\eps_{e+1},\ldots,\eps_r)\iota_{\a(m)}\cdot(\sq(\eps_e,\ldots,\eps_r)\iota_{\a(m)})^{2^{e-1}-1})\nonumber\end{eqnarray}
since $\eps_e=1$. Thus $\chi\sq^{n-k}\io_k\in\im(\sq^1)$ for $k\le\a(m)$.

 Now we consider $k=\a(m)+1$. Let $t_{e-1}=2$, $t_e=0$, else $t_j=\eps_j$, and let $E'$ be the sequence $(t_e,t_{e+1},\ldots)$. Note that $\exc(E')=\a(m)-1$. Then
$$\sq(t_1,\ldots,t_s)\iota_{\a(m)+1}=(\sq(E')\iota_{\a(m)+1})^{2^{e-1}}.$$
We claim that $(\sq(E')\iota_{\a(m)+1})^{2^{e-1}}$ cannot occur as a summand in $\sq^1(M)$ for any monomial $M$ in classes $\sq(R)\iota_{\a(m)+1}$ with $\exc(R)\le\a(m)$. This implies that for  $k=\a(m)+1$, $\chi\sq^{n-k}\iota_k\not\in\im(\sq^1)$ because it contains the term $\sq(t_1,\ldots)\io_k$.

To prove the claim, first note that since $E'$ starts with 0, $(\sq(E')\io_k)^{2^{e-1}}$ cannot be obtained in $\im(\sq^1)$ as in (\ref{s1}). The other feature that keeps it out of $\im(\sq^1)$ is that $k-\exc(E')=2$. This implies that to have $\sq(a_1,\ldots,a_r)\io_k=(\sq(E')\io_k)^{2^p}$ from \ref{Mil}(vi), it must be that $(a_1,\ldots,a_r)=(0^{p-1},2,E')$. This would give
$$(\sq(E')\io_k)^{2^{e-1}}=(\sq(E')\io_k)^{2^{e-1}-2^p}\sq(0^{p-1},2,E'),$$
but this is not in $\im(\sq^1)$ since $\sq(0^{p-1},2,E')\not\in\im(\sq^1)$.

\end{proof}

The following elementary lemma will be useful.
\begin{lem} Let $n=\sum\eps_i 2^i$ with $\eps_i\in\{0,1\}$.\label{lem}
\begin{itemize}
\item[a.] Suppose $n\equiv0\ (4)$ and $\sum r_i(2^i-1)=n-\a(n)-1$ with $r_i\ge0$. Then $\sum r_i\ge\a(n)+1$ with equality if and only if $(r_1,\ldots)$ is obtained from $(\eps_1,\ldots)$ by adding some $(0,\ldots,0,2,-1,0,\ldots)$.
\item[b.] Suppose $n\equiv0\ (8)$ and \begin{equation}\label{dis}\sum r_i(2^i-1)=n-\a(n)-2\end{equation} with $r_i\ge0$. Then $\sum r_i\ge\a(n)+2$ with equality if and only if $(r_1,\ldots)$ is obtained from $(\eps_1,\ldots)$ by two steps of adding some $(0,\ldots,0,2,-1,0,\ldots)$.
\end{itemize}
\end{lem}
\begin{proof} We prove (b), as (a) is similar. Let $n=8m+8$. If $\sum r_i=\a(n)+1$, then, adding this to (\ref{dis}), $8m+7$ has been obtained as the sum of $\a(n)+1$
not-necessarily-distinct 2-powers. Three of those must be used for the 7, so $8m$ is the sum of $(\a(n)-2)$ 2-powers. But $\a(8m)\ge\a(n)-1$, contradiction. A similar contradiction is obtained if $\sum r_i=\a(n)$. If $\sum r_i=\a(n)+2$, then $n$ is obtained as the sum of $\a(n)+2$
not-necessarily-distinct 2-powers. The only way this can be done is by twice splitting some $2^i$ into $2^{i-1}+2^{i-1}$.
\end{proof}

\medskip

\medskip\begin{proof}[Proof of part (d)] Note that for $k=\a(n)$, $\chi\sq^{n-k}\iota_k\in\im(\sq^1)\subset\im(\sq^1,\sq^2)$ by part (c).

Now let $k=\a(n)+1$. Write $n=2^e+2^{e+1}m$ with $m$ even and $m=\sum_{i\ge0}\delta_i2^i$ with $\delta_i\in\{0,1\}$. Let $v=(\delta_0,\delta_1,\ldots)$. Note that $\delta_0=0$. We first show that, mod $\im(\sq^1,\sq^2)$, $\chi\sq^{n-k}\iota_k\equiv(\sq(0,v)\iota_k)^{2^{e-1}}$. To see this, whenever $\delta_i=1$, let $v_i=v+(0^{i-1},2,-1,0,\ldots)$. Then, by Lemma \ref{lem}(a)
$$\chi\sq^{n-k}=\sq(0^{e-2},2,0,v)+\sum_{\delta_i=1}\sq(0^{e-1},1,v_i)+\text{terms of excess}>k.$$
Thus
\begin{eqnarray*}\chi\sq^{n-k}\iota_k&=&\sq(0^{e-2},2,0,v)\iota_k+\sum_{\delta_i=1}\sq(0^{e-1},1,v_i)\iota_k\\
&=&(\sq(0,v)\iota_k)^{2^{e-1}}+\sum_{\delta_i=1}(\sq(1,v_i)\iota_k)^{2^{e-1}}.\end{eqnarray*}
But \begin{equation}\label{disp}(\sq(1,v_i)\iota_k)^{2^{e-1}}=\sq^1(\sq(0,v_i)\iota_k\cdot(\sq(1,v_i)\iota_k)^{2^{e-1}-1})\in\im(\sq^1),\end{equation} proving that $\chi\sq^{n-k}\io_k\equiv(\sq(0,v)\iota_k)^{2^{e-1}}$.

We will complete the proof of part (d) by constructing a homomorphism $\phi:H^n(K(\zt,k);\zt)\to\zt$ such that $\phi(\im(\sq^1,\sq^2))=0$ and $\phi((\sq(0,v)\iota)^{2^{e-1}})=1$. Here and below, we write $\iota$ for $\iota_k$. Let
\begin{eqnarray*}A_1&=&(\sq(0,v)\io)^{2^{e-1}}\\
A_2&=&(\sq(0,v)\io)^{2^{e-1}-4}(\sq(1,v)\io)^2\sq(0,0,v)\io\\
A_3&=&(\sq(0,v)\io)^{2^{e-1}-3}\sq(1,v)\io\cdot\sq(1,0,v)\io\\
A_4&=&(\sq(0,v)\io)^{2^{e-1}-4}\sq(1,0,v)\io\cdot\sq(0,1,v)\io\\
A_5&=&(\sq(0,v)\io)^{2^{e-1}-7}(\sq(1,v)\io)^3\sq(0,0,v)\io\cdot\sq(1,0,v)\io\\
A_6&=&(\sq(0,v)\io)^{2^{e-1}-8}(\sq(1,v)\io)^2\sq(0,0,v)\io\cdot\sq(1,0,v)\io\cdot\sq(0,1,v)\io.\end{eqnarray*}
Here $A_5=0=A_6$ if $e=3$. Then $\phi$ is defined to be the homomorphism which sends the monomials $A_i$ to 1, and all other monomials in the generators $\sq(R)\io$ with $\exc(R)\le\a(m)+1$ to 0.

One can verify that the only way that any of the $A_i$ can occur as a summand of $\sq^1(M)$ or $\sq^2(M)$ for a monomial $M$ of the appropriate degree is as follows, where $\equiv$ is mod the span of all monomials except the $A_i$. Since the number of $A_i$'s in each of these elements of $\im(\sq^1,\sq^2)$ is even, the claim that $\phi(\im(\sq^1,\sq^2))=0$ is proved.
\begin{eqnarray*}\sq^2\bigl((\sq(0,v)\io)^{2^{e-1}-2}\sq(0,0,v)\io\bigr)&\equiv&A_1+A_2\\
\sq^1\bigl((\sq(0,v)\io)^{2^{e-1}-3}\sq(1,v)\io\cdot\sq(0,0,v)\io\bigr)&\equiv&A_2+A_3\\
\sq^1\bigl((\sq(0,v)\io)^{2^{e-1}-4}\sq(0,0,v)\io\cdot\sq(0,1,v)\io\bigr)&\equiv&A_2+A_4\\
\sq^2\bigl((\sq(0,v)\io)^{2^{e-1}-5}\sq(1,v)\io\cdot\sq(0,0,v)\io\cdot\sq(1,0,v)\io\bigr)&\equiv&A_3+A_5\\
\sq^2\bigl((\sq(0,v)\io)^{2^{e-1}-6}\sq(0,0,v)\io\cdot\sq(1,0,v)\io\cdot\sq(0,1,v)\io\bigr)&\equiv&A_4+A_6\\
\sq^1\bigl((\sq(0,v)\io)^{2^{e-1}-7}\sq(1,v)\io\cdot\sq(0,0,v)\io\cdot\sq(1,0,v)\io\cdot\sq(0,1,v)\io\bigr)&\equiv&A_5+A_6.\end{eqnarray*}
The last three are not present when $e=3$.

As an aid for the reader doing this verifying, we note the following relations, using Proposition \ref{Mil}(vi) in the first three. \begin{eqnarray*}\sq^2(\sq(0,0,v)\io)&=&(\sq(0,v)\io)^2\\
\sq^1(\sq(0,1,v)\io)&=&(\sq(1,v)\io)^2\\
\sq^2(\sq(0,v)\io)&=&(\sq(v)\io)^2\\
\sq^2(\sq(1,0,v)\io)&=&\sq(0,1,v)\io.\end{eqnarray*}
Trickier than computing the $\sq^1$ and $\sq^2$ is determining that the $A_i$ cannot be achieved in any other way. For example, you might think that $(\sq(0,v)\io)^2$ as part of the first factor of $A_2$ might be obtained from $\sq^2(\sq(0,0,v)\io)$, but it doesn't occur because it would be coming from $(\sq(0,0,v)\io)^2$ and so would get a coefficient 2.
\end{proof}

\medskip
\begin{proof}[Proof of part (e)] Let $n=3\cdot2^e+2^{e+2}m$ with $m=\sum_{i\ge0}\delta_i2^i$ and $v=(\delta_0,\delta_1,\ldots)$. We will first show that $\chi\sq^{n-k}\io_k\in\im(\sq^1,\sq^2)$ when $k=\a(m)+3$.
Whenever $\delta_i=1$ with $i\ge1$, let $v_i=v+(0^{i-2},2,-1,0,\ldots)$. By Lemma \ref{lem}(a), the only summands of $\chi\sq^{n-k}$ of excess $\le\a(m)+3$ are $\sq(0^{e-2},2,0,1,v)$, $\sq(0^{e-1},3,0,v)$, $\sq(0^{e-1},1,1,v_i)$, and if $\delta_0=1$, $\sq(0^{e-1},1,3,0,\delta_1,\delta_2,\ldots)$.
Then, with $\io=\io_k$,
$$\chi\sq^{n-k}\io=(\sq(0,1,v)\io)^{2^{e-1}}+(\sq(0,v)\io)^{2^e}+\sum_{\delta_i=1}(\sq(1,v_i)\io)^{2^e}+\eps(\sq(3,0,\delta_1,\ldots)\io)^{2^e}.$$
Mod $\im(\sq^1)$, this equals $Y^{2^{e-1}}$, where $Y=\sq(0,1,v)\io+(\sq(0,v)\io)^2$, since the  terms after the first two are in $\im(\sq^1)$, similarly to (\ref{disp}). This $Y$ is a generalization of $\sq^2\sq^1$, and satisfies $\sq^2Y=0$, $\sq^1(Y)=\sq(1,1,v)\io$, and $\sq^2(\sq(1,0,v)\io)=Y$.
Let
\begin{eqnarray*}B_1&=&Y^{2^{e-1}-3}\sq(1,0,v)\io(\sq(1,1,v)\io)^2\\
B_2&=&Y^{2^{e-1}-4}\sq(0,0,v)\io(\sq(1,1,v)\io)^3\\
B_3&=&Y^{2^{e-1}-4}(\sq(2,0,v)\io)^2(\sq(1,1,v)\io)^2\\
B_4&=&Y^{2^{e-1}-4}(\sq(0,v)\io)^2\sq(1,0,v)\io(\sq(1,1,v)\io)^2.\end{eqnarray*}
One can verify the following equations. Summing them yields the desired conclusion, $Y^{2^{e-1}}\in\im(\sq^1,\sq^2)$.
\begin{eqnarray*}\sq^2\bigl(Y^{2^{e-1}-1}\sq(1,0,v)\io\bigr)&=&Y^{2^{e-1}}+B_1\\
\sq^1\bigl(Y^{2^{e-1}-3}\sq(0,0,v)\io(\sq(1,1,v)\io)^2\bigr)&=&B_1+B_2\\
\sq^2\bigl(Y^{2^{e-1}-4}\sq(0,0,v)\io\cdot\sq(2,0,v)\io(\sq(1,1,v)\io)^2\bigr)&=&B_2+B_3+B_4\\
\sq^1\bigl(Y^{2^{e-1}-3}(\sq(2,0,v)\io)^2\sq(1,1,v)\io\bigr)&=&B_3\\
\sq^1\bigl(Y^{2^{e-1}-4}(\sq(0,v)\io)^2\sq(0,0,v)\io(\sq(1,1,v)\io)^2\bigr)&=&B_4.\end{eqnarray*}

Again let $n=3\cdot2^e+2^{e+2}m$ with $m=\sum_{i\ge0}\delta_i2^i$ and $v=(\delta_0,\delta_1,\ldots)$. We will now show that $\chi\sq^{n-k}\io_k\not\in(\sq^1,\sq^2)$ when $k=\a(m)+4=\a(n)+2$. By Lemma \ref{lem}(b), $\chi\sq^{n-k}$ has many summands of excess $k$ (and none with smaller excess). Letting $v'$ denote $v$ with the addition of one $(\ldots,0,2,-1,0,\ldots)$, $v''$ obtained from $v$ by two such additions, and $v_0$ being $v$ with $\delta_0=1$ changed to $\delta_0=0$, we list these now.
\begin{eqnarray*}\sq(0^{e-2},2,2,0,v)\io_k&=&(\sq(2,0,v)\io_k)^{2^{e-1}}\\
\sq(0^{e-3},2,1,0,1,v)\io_k&=&(\sq(1,0,1,v)\io_k)^{2^{e-2}}\\
\sq(0^{e-2},2,0,1,v')\io_k&=&(\sq(0,1,v')\io_k)^{2^{e-1}}\\
\sq(0^{e-1},3,0,v')\io_k&=&(\sq(0,v')\io_k)^{2^e}\\
\sq(0^{e-1},1,1,v'')\io_k&=&(\sq(1,v'')\io_k)^{2^e}\\
\sq(0^{e-2},2,0,3,v_0)\io_k&=&(\sq(0,3,v_0)\io_k)^{2^{e-1}}\\
\sq(0^{e-1},3,2,v_0)\io_k&=&(\sq(2,v_0)\io_k)^{2^e}.\end{eqnarray*}

Similarly to the proof of part (d), we will construct a homomorphism $\phi$ from $H^n(K(\zt,k);\zt)$ to $\zt$ sending $(\sq(2,0,v)\io_k)^{2^{e-1}}$ and nine other specified monomials to 1, and all others to 0, and annihilating $(\im(\sq^1,\sq^2))$. The above monomials other than the first are sent to 0, so we need not worry about them.

We will take some notational shortcuts, writing $(r_1,r_2)^p$ for $\sq(r_1,r_2,0,v)\iota_k)^p$, and similarly with $r_2$ omitted. The ten monomials $C_i$ that are mapped to 1 by $\phi$ are listed below.
\begin{eqnarray*} C_1&=&(2)^{2^{e-1}}\\
C_2&=&(2)^{2^{e-1}-4}(3)^2(0,2)\\
C_3&=&(2)^{2^{e-1}-4}(1,2)(0,3)\\
C_4&=&(2)^{2^{e-1}-3}(3)(1,2)\\
C_5&=&(0)(2)^{2^{e-1}-4}(3)(0,3)\\
C_6&=&(0)(2)^{2^{e-1}-3}(3)^2\\
C_7&=&(2)^{2^{e-1}-8}(3)^2(0,2)(1,2)(0,3)\\
C_8&=&(2)^{2^{e-1}-7}(3)^3(0,2)(1,2)\\
C_9&=&(0)(2)^{2^{e-1}-8}(3)^3(0,2)(0,3)\\
C_{10}&=&(0)(2)^{2^{e-1}-7}(3)^2(1,2)(0,3).
\end{eqnarray*}
Note that $C_7$ through $C_{10}$ are only present for $e\ge4$. The only relations involving $\sq^1(M)$ or $\sq^2(M)$ involving any of the $C_i$ are as follows, where again $\equiv$ is mod monomials which are not one of our $C_i$.
\begin{eqnarray*}\sq^1\bigl((2)^{2^{e-1}-4}(0,2)(0,3)\bigr)&=&C_2+C_3\\
\sq^1\bigl((2)^{2^{e-1}-3}(3)(0,2)\bigr)&=&C_2+C_4\\
\sq^1\bigl((0)(2)^{2^{e-1}-3}(0,3)\bigr)&\equiv&C_5+C_6\\
\sq^2\bigl((2)^{2^{e-1}-2}(0,2)\bigr)&=&C_1+C_2\\
\sq^2\bigl((0)(2)^{2^{e-1}-4}(3)(1,2)\bigr)&\equiv&C_4+C_5\\
\sq^2\bigl((0)(2)^{2^{e-1}-1}\bigr)&\equiv&C_1+C_{6}\\
\sq^1\bigl((2)^{2^{e-1}-7}(3)(0,2)(1,2)(0,3)\bigr)&\equiv&C_7+C_8\\
\sq^1\bigl((0)(2)^{2^{e-1}-7}(3)^2(0,2)(0,3)\bigr)&\equiv&C_9+C_{10}\\
\sq^2\bigl((2)^{2^{e-1}-6}(0,2)(1,2)(0,3)\bigr)&\equiv&C_3+C_7\\
\sq^2\bigl((2)^{2^{e-1}-5}(3)(0,2)(1,2)\bigr)&\equiv&C_4+C_8\\
\sq^2\bigl((0)(2)^{2^{e-1}-8}(3)^3(0,2)(1,2)\bigr)&\equiv&C_8+C_9\\
\sq^2\bigl((0)(2)^{2^{e-1}-6}(3)(0,2)(0,3)\bigr)&\equiv&C_5+C_9\\
\sq^2\bigl((0)(2)^{2^{e-1}-5}(1,2)(0,3)\bigr)&\equiv&C_3+C_{10}\\
\sq^2\bigl((0)(2)^{2^{e-1}-5}(3)^2(0,2)\bigr)&\equiv&C_2+C_6\\
\sq^2\bigl((0)(2)^{2^{e-1}-9}(3)^2(0,2)(1,2)(0,3)\bigr)&\equiv&C_7+C_{10}.
\end{eqnarray*}
Relations 7 though 14 are only relevant for $e\ge4$, and the last one for $e\ge5$. Some relations useful in the analysis are, in our shorthand notation, $\sq^1(0,3)=(3)^2$, $\sq^2(0,2)=(2)^2$, and $\sq^2(1,2)=(0,3)$.

Since the only elements of $\im(\sq^1,\sq^2)$ which involve any $C_i$ involve an even number of $C_i$, we conclude that $\phi(\im(\sq^1,\sq^2))=0$.
\end{proof}

\section{Existence of manifolds, I}\label{ifsec}
We begin this section by presenting a proof of the ``if'' part of Theorem \ref{or}. By Theorem \ref{genlthm}(a), we must show that, for $k$ as in part (i) of Theorem \ref{chithm}, a mod-2 homology class dual to $\chi\sq^{n-k}\io_k$ is the reduction of an integral class.

For $n=4m+b$ with $1\le b\le3$ and $k=\a(m)+b$, similarly to the first part of the proof of part (b) of Theorem \ref{chithm}, $\chi\sq^{n-k}\io_k$ contains the term $\sq(0,\eps_2,\ldots,\eps_r)$, and so $\sq^1\chi\sq^{n-k}\io_k\ne0$. This implies that a dual mod-2 homology class is the reduction of an integral class since the composite
$$H_{n+1}(X;\zt)\mapright{\partial} H_n(X;\Z)\mapright{\rho_2}H_n(X;\zt)$$
is dual to $\sq^1$.

If $n=2^eu$ with $u$ odd and $e\ge2$, and $k=\a(n)+1$, then, by the proof of part (c), $\chi\sq^{n-k}\io_k=(\sq(E')\io_k)^{2^{e-1}}$ where $\exc(E')=k-2$ and the first entry of $E'$ is 0. Let $x=\sq(E')\io_k$.
 In \cite[Theorem 5.5]{Br} or \cite[Theorem 1.3.2]{Cl}, it is shown that for such a class $x$ (even-dimensional primitive with $\sq^1x\ne0$), $d_e(x^{2^{e-1}})\ne0$ for all $e$ in the cohomology Bockstein spectral sequence, and then, by \cite[Theorem 4.7]{Br} or \cite[Theorem 2.4.4]{Cl},  this implies that an  integral homology class dual to $x^{2^{e-1}}$ has order $2^e$. This completes the proof of the ``if'' part of Theorem \ref{or}.

Next we prove  Theorem \ref{spinnonimm} for $n\le15$. Recall from Theorem \ref{genlthm} that we need that $\chi\sq^{n-k}\io_k\not\in\im(\sq^1,\sq^2)$ and a dual class is in the image from $ko_n(K(\zt,k))$.

 The $n\le7$ result can be seen from the fact that elements $\sq(R)$ not in $\im(\sq^1,\sq^2)$ satisfy $|\sq(R)|+\exc(R)\ge8$, so if $\sq(R)\io_k\not\in\im(\sq^1,\sq^2)$, then $n=|\sq(R)|+k\ge8$.

For $n=8$, the smallest possible value of $k$ in Theorem \ref{chithm}(ii) is 2, while
for $9\le n\le15$, it is $k=n-7$. Detailed Adams spectral sequence  (ASS) calculations, discussed below, show that in the ASS converging to $ko_*(K(\zt,k))$,
$\chi\sq^6\io_2$ is a permanent cycle, so yields the desired element in $ko_8(K(\zt,2))$, while for $9\le n\le12$,
$\chi\sq^7\io_{k}$ supports a nonzero $d_2$-differential for $2\le k\le5$, but not for $6\le k\le8$. For $9\le n\le 12$, we next try $\chi\sq^6\io_k$ with $k=n-6$,  and it is clear from Figure \ref{fig2} that there are no possible differentials on this class when $k=3$ ($n=9$) and hence also not for larger values of $k$. Once we have verified these claims,  Theorem \ref{spinnonimm} follows for $n\le 15$.

The $E_2$-term of the ASS converging to $ko_*(K(\zt,k))$ is $\ext_{A_1}(H^*(K(\zt,k);\zt),\zt)$, where $A_1$ is generated by $\sq^1$ and $\sq^2$. For $2\le k\le 6$ and $*\le k+8$, these are shown in Figures \ref{fig1}, \ref{fig2}, and \ref{fig3}. These were obtained by calculating minimal resolutions of $H^*(K(\zt,k);\zt)$ as $A_1$-modules. See, e.g., \cite[pp.~121--125]{W}. The classes involved in the key $d_2$-differentials are circled.

\bigskip

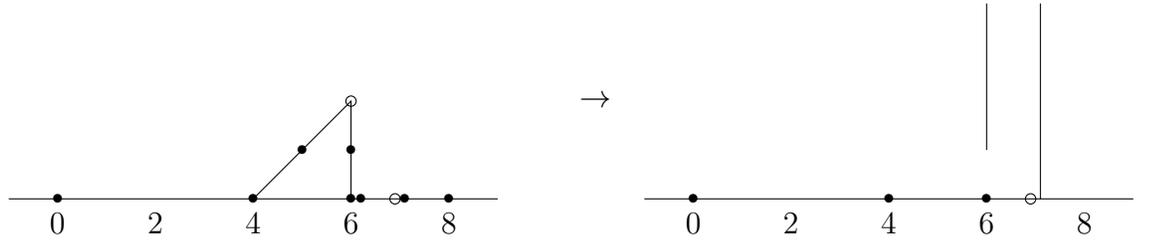
\def\sb{$\ssize\bullet$}
\bigskip
\begin{minipage}{6in}
\begin{fig}\label{fig1}

{\bf $ko_{*+2}(K(\zt,2)) \ \to\ H_{*+2}(K(\zt,2);\Z)$}

\begin{center}

\begin{\tz}[scale=.65]
\node at (0,-.5) {$0$};
\node at (2,-.5) {$2$};
\node at (4,-.5) {$4$};
\node at (6,-.5) {$6$};
\node at (8,-.5) {$8$};
\node at (0,0) {\sb};
\draw (2,1) -- (2,4);
\draw (3,0) -- (3,4);
\node at (1.6,2) {$A$};
\node at (2.6,0) {$B$};
\node at (5.5,2) {$C$};
\node at (6.9,-.4) {$D$};
\node at (4,0) {\sb};
\draw (7.9,4) -- (5.9,2) -- (5.9,4);
\draw (8.1,2) -- (6.1,0) -- (6.1,4);
\draw (8.9,2) -- (6.9,0) -- (6.9,4);
\draw (7.1,0) -- (7.1,4);
\node at (2,1) {\sb};
\node at (2,2) {\sb};
\node at (2,3) {\sb};
\node at ( 3,2) {\sb};
\node at (5.9,3 ) {\sb};
\node at (6.8,2.9 ) {\sb};
\node at ( 6.9,1.1) {\sb};
\node at (7.9,4 ) {\sb};
\node at (7.9,1 ) {\sb};
\node at (8.9,2 ) {\sb};
\node at (6.1,0 ) {\sb};
\node at (7.2,1.1 ) {\sb};
\node at ( 8.1,2) {\sb};
\node at (3,0) {\sb};
\node at (3,1) {\sb};
\node at (15,0) {\sb};
\node at (15,1) {\sb};
\node at (15,2) {\sb};
\node at (15,3 ) {\sb};
\node at (16,0) {\sb};
\node at (16,1 ) {\sb};
\node at (7.9,0) {\sb};
\node at (8.1,0) {\sb};
\draw (6.9,0) circle (3pt);
\draw (5.9,2) circle (3pt);
\node at (11,2) {$\to$};
\node at (13,-.5) {$0$};
\node at (15,-.5) {$2$};
\node at (17,-.5) {$4$};
\node at (13,0) {\sb};
\draw (15,0) -- (15,4);
\draw (15,2) -- (16,0) -- (16,4);
\draw (15,3) -- (16,1);
\node at (17,0) {\sb};
\draw (-1,0) -- (2.2,0);
\draw (3.2,0) -- (6.5,0);
\draw (7.5,0) -- (9,0);
\draw (12.5,0) -- (17.5,0);
\end{\tz}
\end{center}
\end{fig}
\end{minipage}

\bigskip

\def\sb{$\ssize\bullet$}
\bigskip
\begin{minipage}{6in}
\begin{fig}\label{fig2}

{\bf $ko_{*+3}(K(\zt,3))\ \to \ ko_{*+4}(K(\zt,4))$}

\begin{center}

\begin{\tz}[scale=.65]
\node at (0,-.5) {$0$};
\node at (2,-.5) {$2$};
\node at (4,-.5) {$4$};
\node at (6,-.5) {$6$};
\node at (8,-.5) {$8$};
\node at (13,-.5) {$0$};
\node at (15,-.5) {$2$};
\node at (17,-.5) {$4$};
\node at (19,-.5) {$6$};
\node at (21,-.5) {$8$};
\node at (4,0) {\sb};
\node at (4,1) {\sb};
\node at (0,0) {\sb};
\draw (4,0) -- (4,1);
\draw (6,0) -- (7,1);
\draw (7,0) -- (8,1);
\draw (6,2) -- (7,3);
\draw (7,0) circle (3pt);
\draw (6,2) circle (3pt);
\node at (5,0) {\sb};
\node at (13,0) {\sb};
\node at ( 6,0) {\sb};
\node at (7,1 ) {\sb};
\node at (8,1 ) {\sb};
\node at (7,3 ) {\sb};
\node at (11,2) {$\to$};
\draw (17,4) -- (17,0) -- (19,2);
\node at ( 17,0) {\sb};
\node at ( 17,1) {\sb};
\node at (17,2 ) {\sb};
\node at (17.8,.8 ) {\sb};
\draw (18,0) -- (18,4);
\draw (19,2) circle (3pt);
\node at (18.9,0) {\sb};
\node at (19.2,0) {\sb};
\draw (20,0) circle (3pt);
\draw (21,1) -- (21,4);
\draw (20.8,3) -- (20.8,4.5);
\draw (-1,0) -- (9,0);
\draw (12,0) -- (22,0);
\end{\tz}
\end{center}
\end{fig}
\end{minipage}

\bigskip

\def\sb{$\ssize\bullet$}
\bigskip
\begin{minipage}{6in}
\begin{fig}\label{fig3}

{\bf $ko_{*+5}(K(\zt,5))\ \to \ ko_{*+6}(K(\zt,6))$}

\begin{center}

\begin{\tz}[scale=.65]
\node at (0,-.5) {$0$};
\node at (2,-.5) {$2$};
\node at (4,-.5) {$4$};
\node at (6,-.5) {$6$};
\node at (8,-.5) {$8$};
\node at (13,-.5) {$0$};
\node at (15,-.5) {$2$};
\node at (17,-.5) {$4$};
\node at (19,-.5) {$6$};
\node at (21,-.5) {$8$};
\node at (4,0) {\sb};
\node at (5,1) {\sb};
\node at (6,0) {\sb};
\node at (6,1) {\sb};
\node at (0,0) {\sb};
\draw (4,0) -- (6,2)  -- (6,0);
\node at (6.2,0) {\sb};
\node at (7.1,0) {\sb};
\draw (6.9,0) circle (3pt);
\draw (6,2) circle (3pt);
\node at (8,0) {\sb};
\node at (11,2) {$\to$};
\node at (13,0) {\sb};
\node at (17,0) {\sb};
\node at (19,0) {\sb};
\draw (19.9,0) circle (3pt);
\draw (19,1) -- (19,4);
\draw (20.1,0) -- (20.1,4);
\draw (-1,0) -- (9,0);
\draw (12,0) -- (22,0);
\end{\tz}
\end{center}
\end{fig}
\end{minipage}

\bigskip
We establish the differential when $k=2$, and use morphisms of minimal resolutions to see that the circled classes map as indicated as $k$ increases.
For $k=2$, we use the morphism $ko_*(K(\zt,2))\to H_*(K(\zt,2);\Z)$. This is  depicted in Figure \ref{fig1}.  The $d_2$-differential in ASS$(H_*(K(\zt,2);\Z))$ is implied by results of \cite{Br} or \cite{Cl} used earlier. This implies  $d_2(A)=B$ (not pictured) in ASS($ko_*(K(\zt,2))$). We show below that the $d_2$-differential from $C$ to $D$ is implied by the action of the ASS of $bo_*$ on that of $ko_*(K(\zt,2))$.

Let $\tau$ (resp.~$h_0$) denote the element of $E_2(bo)$ corresponding to the filtration-3 generator of $\pi_4(bo)$ (resp.~$2$). Then $\tau\cdot A=h_0^3C$ and $\tau\cdot B=h_0^3D$. This can be seen from the minimal resolutions. Thus
$$h_0^3d_2(D)=d_2(h_0^3D)=d_2(\tau B)=\tau\cdot d_2(B)=\tau\cdot A=h_0^3C,$$
so $d_2(D)=C$.

This determination of $ko_*(K(\zt,2))$ was done, in a similar manner but a somewhat different context, in \cite{W}. Many of our deductions here for other $ko_*(K(\zt,k))$ were also made there, using a different argument.

\section{Existence of manifolds, II.}
In this section, we prove
Theorem \ref{spinnonimm} for $n>15$. Let $K_k=K(\zt,k)$ and $\ext_B(X)=\ext_B(H^*X,\zt)$. Here $B=A_1$ or $E_1$, the latter being the exterior algebra on the Milnor primitives $Q_0=\sq^1$ and $Q_1=\sq(0,1)$. The $E_2$-term of the ASS converging to $ku_*(X)$ is $\ext_{E_1}(X)$, and there is a nice morphism $ko_*(X)\to ku_*(X)$. Theorem \ref{spinnonimm} for $n>15$ follows from Theorem \ref{genlthm}(b), Theorem \ref{chithm}(ii), and the following result, the proof of which requires detailed ASS calculations.
\begin{thm}\label{ASS}
\begin{itemize}
\mbox{}
\item[i.] For $n=16$ and $32$, the element of $\ext_{A_1}^{0,n}(K_2)$ dual to $\chi\sq^{n-2}\io_2$ is a permanent cycle in the ASS converging to $ko_*(K_2)$.
\item[ii.] For $n=17$ (resp.~$33$), the element of $\ext_{E_1}^{0,n}(K_2)$ dual to $\chi\sq^{n-2}\io_2$ supports a nonzero $d_4$ (resp.~$d_8$) differential in the ASS converging to $ku_*(K_2)$.
\item[iii.] The element of $\ext_{A_1}^{0,18}(K_3)$ dual to $\chi\sq^{15}\io_3$ is a permanent cycle in the ASS converging to $ko_*(K_3)$.
\end{itemize}
\end{thm}

Part (ii) implies the analogous result for $ko_*$ since the morphism $ko_*X\to ku_*X$ is induced by a morphism of spectral sequences.
That elements dual to $\chi\sq^{14}\io_3$, $\chi\sq^{30}\io_3$, and  $\chi\sq^{15}\io_k$ for $4\le k\le8$ are permanent cycles  follows from parts (i) and (iii) by naturality.

The remainder of the paper is devoted to proving Theorem \ref{ASS}. In \cite[ Section 5]{W}, the second author computed $\ext_{A_1}(K_2)$ through dimension 36. An incorrect deduction was made regarding some differentials in this ASS around dimension 33, but we have verified that its $A_1$-module splitting and determination of associated Ext groups is correct. Although not explicitly noted there, one can read off that $\ext_{A_1}^{s,t}(K_2)=0$ for $s>0$, $t-s\equiv7\pmod8$, $t-s<39$. This is all that is required for our Theorem \ref{ASS}[i.].

For part (ii), we give the complete calculation of the ASS for $ku_*(K_2)$ through dimension 34, except for filtration-0 $\zt$'s corresponding to free $E_1$ summands. In this range, $H^*(K_2;\zt)$ is a polynomial algebra on classes $u_2=\io_2$, $u_3=\sq^1\io$, $u_5=\sq^{2,1}\io$, $u_9=\sq^{4,2,1}\io$, $u_{17}=\sq^{8,4,2,1}\io$, and $u_{33}=\sq^{16,8,4,2,1}\io$.
The $E_1$ action is given in Table \ref{T2}.

\begin{table}[h]
\caption{$E_1$ action on generators of $H^*(K_2)$}
\label{T2}

\begin{tabular}{c|cccccc}
$x$&$u_2$&$u_3$&$u_5$&$u_9$&$u_{17}$&$u_{33}$\\
\hline
$Q_0x$&$u_3$&$0$&$u_3^2$&$u_5^2$&$u_9^2$&$u_{17}^2$\\
$Q_1x$&$u_5$&$u_3^2$&$0$&$u_3^4$&$u_5^4$&$u_9^4$
\end{tabular}
\end{table}

With $P$ (resp.~$E$) denoting a polynomial (resp.~exterior) algebra, in this range the $Q_0$-homology is $P[u_2^2]\otimes E[x_5]$, where $x_5=u_5+u_2u_3$, and $Q_1$-homology is
$$P[u_2^2]\otimes E[x_9,x_{17},u_9^2, u_{17}^2],$$
where $x_9=u_9+u_3^3$ and $x_{17}=u_{17}+u_2u_5^3$. There is an $E_1$-submodule $N$ with a single nonzero element in gradings 5, 7, 8, 9, 10, with generators $x_5$, $x_7=u_2u_5$, and $x_9$, with $Q_0x_7=Q_1x_5$ and $Q_1x_7=Q_0x_9$. It has a $Q_0$-homology class $x_5$, and a $Q_1$-homology class $x_9$. The beginning of the $E_2$-term for $\langle u_2^{2i+2}\rangle\oplus u_2^{2i}N$ is depicted in Figure \ref{figa}.

\bigskip
\begin{minipage}{6in}
\begin{fig}\label{figa}

{\bf $\ext_{E_1}(\langle u_2^{2i+2}\rangle\oplus u_2^{2i}N)$}

\begin{center}

\begin{\tz}[scale=.65]
\node at (2,-.5) {$4i+$};
\node at (10,-.5) {$10$};
\node at (4,-.5) {$4$};
\node at (6,-.5) {$6$};
\node at (8,-.5) {$8$};
\draw (4,0) -- (4,6) ;
\draw [color=red] (5,0) -- (5,6);
\draw (6,1) -- (6,6);
\draw [color=red] (7,0) -- (7,6);
\draw (8,2) -- (8,6);
\draw [color=red] (9,1) -- (9,6);
\draw (10,3) -- (10,6);
\draw [color=red] (11,2) -- (11,6);
\draw (1,0) -- (12,0);
\end{\tz}
\end{center}
\end{fig}
\end{minipage}

\bigskip
Comparison with the results for $H_*(K_2;\Z)$ in \cite{Br} cited early in Section \ref{ifsec} shows that there is a $d_{\nu(4i+4)}$-differential between the first pair of towers in this chart, where $\nu(-)$ denotes the exponent of 2 in an integer. This differential is promulgated in each chart by the action of $v_1\in E_2^{1,3}(bu)$. The $v_1$-periodic classes remaining after removing classes involved in these differentials are, in the range being considered here,  $v_1$-towers on $u_2^4$, $u_2^8$, $h_0u_2^8$, $u_2^{12}$, $u_2^{16}$, $h_0u_2^{16}$, and $h_0^2u_2^{16}$. These will appear as lines of slope 1/2 in Figure \ref{figb}. Here $h_0$ is the Ext element corresponding to multiplication by 2. We are abusing notation here by writing a cohomology class to denote an Ext class dual to it.

The submodules $u_2^{2i+2}$ and $u_2^{2i}N$ account for all of the $Q_0$-homology of $H^*K_2$. Through grading 35, the remaining $Q_1$-homology classes are $$P[u_2^2]\otimes E[x_9]\otimes\langle x_{17},u_9^2, u_{17}^2,u_9^2x_{17}\rangle.$$ Let $x_{33}=u_{33}+u_2u_3u_5^2u_9^2$. There are $Q_0$-free $E_1$-submodules $M_4$ and $M_5$ such that $M_4$ has a single nonzero class in gradings 17 and 18, and $M_5$ in gradings 33, 34, 35, and 36, realizing the $Q_1$-homology classes $x_{17}$, $u_9^2$, $u_{17}^2$, and $u_9^2x_{17}$, and beginning with $x_{17}$ and $x_{33}$, respectively. Then the inclusion of the $E_1$-submodule $$P[u_2^2]\otimes(\langle 1\rangle\oplus N\oplus M_4\oplus(N\otimes M_4)\oplus M_5)$$ into $H^*(K_2)$ induces an isomorphism in $Q_0$- and $Q_1$-homology through dimension 42, and hence an isomorphism in $\ext_{E_1}$ above filtration 0 through roughly the same range. For any $Q_0$-free $E_1$-module $M$, $M\otimes N$ and $x_9M$ have isomorphic $\ext_{E_1}$ in positive filtration. $\ext_{E_1}(M_4)$ is a single $v_1$-tower beginning in grading 17, while $\ext_{E_1}(M_5)$ has $v_1$-towers beginning in 33 and 35, connected by $h_0$.

The initial differential implied by integral homology was $d_2(x_9)=v_1^2u_2^2$. The derivation property of differentials implies that $d_2(u_2^{4i}x_9x_{17})=v_1^2u_2^{4i+2}x_{17}$.
Listing only $v_1$-periodic classes, the elements remaining after the above considerations are depicted in Figure \ref{figb}.

\bigskip
\begin{minipage}{6in}
\begin{fig}\label{figb}

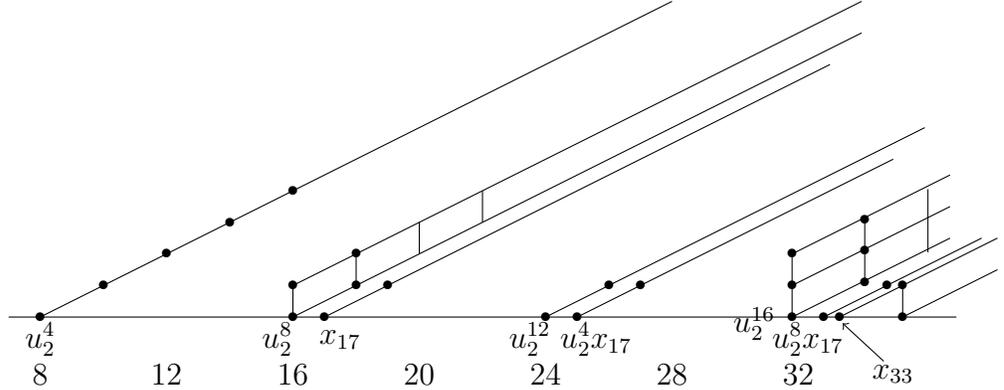
{\bf $v_1$-periodic classes in part of ASS for $ku_*(K_2)$}

\begin{center}

\begin{\tz}[scale=.42]
\node at (8,-1.8) {$8$};
\node at (12,-1.8) {$12$};
\node at (16,-1.8) {$16$};
\node at (20,-1.8) {$20$};
\node at (24,-1.8) {$24$};
\node at (28,-1.8) {$28$};
\node at (32,-1.8) {$32$};
\node at (8,-.7) {$u_2^4$};
\draw (8,0) -- (28,10);
\node at (8,0) {\sb};
\node at (10,1) {\sb};
\node at (12,2) {\sb};
\node at (15.5,-.7) {$u_2^8$};
\draw (16,0) -- (34,9);
\draw (16,1) -- (34,10);
\draw (16,0) -- (16,1);
\draw (18,1) -- (18,2);
\draw (20,2) -- (20,3);
\node at (14,3) {\sb};
\node at (16,4) {\sb};
\node at (16,0) {\sb};
\node at (16,1) {\sb};
\node at (18,1) {\sb};
\node at (18,2) {\sb};
\draw (22,3) -- (22,4);
\node at (17.5,-.7) {$x_{17}$};
\draw (17,0) -- (33,8);
\node at (17,0) {\sb};
\node at (19,1) {\sb};
\node at (23.5,-.7) {$u_2^{12}$};
\draw (24,0) -- (36,6);
\node at (24,0) {\sb};
\node at (26,1) {\sb};
\node at (25.6,-.7) {$u_2^4x_{17}$};
\draw (25,0) -- (35,5);
\node at (25,0) {\sb};
\node at (27,1) {\sb};
\draw (31.8,0) -- (31.8,2);
\node at (31.8,0) {\sb};
\node at (31.8,1) {\sb};
\node at (31.8,2) {\sb};
\node at (34.1,1.1) {\sb};
\node at (34.1,2.1) {\sb};
\node at (34.1,3.1) {\sb};
\draw (34.1,1.05) -- (34.1,3.05);
\draw (36.1,2.05) -- (36.1,4.05);
\draw (31.8,0) -- (36.8,2.5);
\draw (31.8,1) -- (36.8,3.5);
\draw (31.8,2) -- (36.8,4.5);
\node at (30.6,-.2) {$u_2^{16}$};
\node at (32.3,-.7) {$u_2^8x_{17}$};
\draw (32.8,0) -- (37.8,2.5);
\draw (33.3,0) -- (38.3,2.5);
\node at (35,-1.8) {$x_{33}$};
\draw [->] (34.7,-1.4) -- (33.4,-.2);
\draw (35.3,0) -- (35.3,1);
\draw (35.3,0) -- (38.3,1.5);
\node at (33.3,0) {\sb};
\node at (35.3,1) {\sb};
\node at (35.3,0) {\sb};
\node at (32.8,0) {\sb};
\node at (34.8,1) {\sb};
\draw (7,0) -- (37,0);
\end{\tz}
\end{center}
\end{fig}
\end{minipage}

\bigskip
We claim that $d_4(x_{17})=v_1^4u_2^4$. To see this, let $f:CP^\infty\to CP^\infty$ denote the $H$-space squaring map, and $g:CP^\infty\to K_2$ correspond to the nonzero element of $H^2(CP^\infty;\zt)$. The composite $g\circ f$ is trivial, and so $g_*:ku_*(CP^\infty)\to ku_*(K_2)$ sends all elements in $\im(ku_*(CP^\infty)\mapright{f_*}ku_*(CP^\infty))$ to 0. Let $\b_i\in ku_{2i}(CP^\infty)$ be dual to $y^i$, where $y$ generates $ku^2(CP^\infty)$. The $[2]$-series for $ku$ is $2x+v_1x^2$, and it follows from \cite[Theorem 3.4]{RW} that
$f_*(\b_j)$ equals the coefficient of $x^j$ in $\dsum_{i\ge1}\b_i(v_1x^2+2x)^i$. Letting $j=8$, we obtain that the following element maps to 0 in $ku_*(K_2)$:
$$v_1^4\b_4+40v_1^3\b_5+240v_1^2\b_6+448v_1\b_7+2^8\b_8.$$
All classes except the first map to 0 in $ku_*(K_2)$. Since $g_*(\b_4)=u_2^4$, we deduce that $v_1^4u_2^4=0$ in $ku_*(K_2)$. The only way that this can occur is by the asserted $d_4$-differential.
By the derivation property, $d_4(u_2^8x_{17})=v_1^4u_2^{12}$.

Similarly to this, using $CP^\infty$, we obtain that $v_1^8u_2^8=0$ in $ku_*(K_2)$. The only way that this can happen is with $d_5(u_2^4x_{17})=h_0v_1^4u_2^8$ and $d_8(x_{33})=v_1^8u_2^8$. Since the Ext class $x_{33}$ evaluates nontrivially on $\chi\sq^{31}\io_2$, this completes the proof of part (ii) of Theorem \ref{ASS}.

We will determine the $ko$-homology of $K(\zt,3)$ through grading 20, providing more detail than we did in the smaller range of dimensions considered in Section \ref{ifsec}. Through dimension 24, $H^*(K_3;\zt)$ is a polynomial algebra on the generators listed in Table \ref{T5}.
\vfill\eject

\begin{table}[h]
\caption{Generators of $H^*(K_3;\zt)$}
\label{T5}

\begin{tabular}{cc|ccc}
&$x$&$\sq^1x$&$\sq^2x$&$Q_1x$\\
\hline
$g_3$&$\io_3$&$g_4$&$g_5$&$g_6+g_3^2$\\
$g_4$&$\sq^1\io$&$0$&$g_6$&$g_7$\\
$g_5$&$\sq^2\io$&$g_3^2$&$g_7$&$g_4^2$\\
$g_6$&$\sq^{2,1}\io$&$g_7$&$0$&$0$\\
$g_7$&$\sq^{3,1}\io$&$0$&$0$&$0$\\
$g_9$&$\sq^{4,2}\io$&$g_5^2$&$0$&$g_3^4$\\
$g_{10}$&$\sq^{4,2,1}\io$&$g_{11}$&$g_6^2$&$g_{13}$\\
$g_{11}$&$\sq^{5,2,1}\io$&$0$&$g_{13}$&$g_7^2$\\
$g_{13}$&$\sq^{6,3,1}\io$&$g_7^2$&$0$&$0$\\
$g_{17}$&$\sq^{8,4,2}\io$&$g_9^2$&$0$&$g_5^4$\\
$g_{18}$&$\sq^{8,4,2,1}\io$&$g_{19}$&$g_{10}^2$&$g_{21}$\\
$g_{19}$&$\sq^{9,4,2,1}\io$&$0$&$g_{21}$&$g_{11}^2$\\
$g_{21}$&$\sq^{10,5,2,1}\io$&$g_{11}^2$&$0$&$0$
\end{tabular}
\end{table}

\bigskip
From this, one readily determines that through grading 20 the $Q_0$-homology classes are $g_6^2$, $g_{13}'=g_{13}+g_6g_7$, and $g_{10}^2$, while $Q_1$-homology classes are $g_3^2$, $g_5^2$, $g_{11}'=g_{11}+g_4g_7$, $g_3^2g_5^2$, $g_3^2g_{11}'$, $g_9^2$, and $g_{10}^2$. We also let $g_{10}'=g_{10}+g_4g_6$.

In Table \ref{T6}, we list eight $A_1$-submodules $M_i$ whose direct sum carries exactly the $Q_0$- and $Q_1$-homology of $H^*(K_3)$ through grading 20. Thus the inclusion of this sum into $H^*(K_3)$ induces an isomorphism in $\ext^{s,t}_{A_1}$ for $s>0$ in this range. We just list the $A_1$-generators of the modules. In Figures \ref{figc} and \ref{figd} we will depict $\ext_{A_1}(M_i)$. The subscript of $M_i$ is the grading of the bottom class. The chart for the second of each pair of summands appears in red.
For $i=12$ and 13, $x_i$ generates a free $A_1$-submodule but is necessary for inclusion since $\sq^{2,1}x_{i+3}=\sq^{2,2,2}x_i$. Some of the modules can be extended beyond grading 22 by adding higher generators. In Table \ref{T6}, $x_{19}=g_4^2g_5g_6+g_3g_4g_6^2+g_3g_4^4$.
We have included $M_{21}$ because its Ext impacts that of $M_{18}$.

\begin{table}[h]
\caption{Submodules of $H^*(K_3)$}
\label{T6}

\begin{tabular}{c|l|cc}
$i$&$A_1$-generators of $M_i$&$H_*(-;Q_0)$&$H_*(-;Q_1)$\\
\hline
$3$&$g_3$,\ $g_3g_4$&$g_6^2$&$g_3^2$\\
$9$&$g_9$,\ $g_3^2g_5$,\  $g_3g_4^3$,\  $x_{19}$&$0$&$g_5^2$\\
$10$&$g_{10}'$&$g_{13}'$&$g_{11}'$\\
$12$&$g_3g_9$,\  $g_5^3$,\ $g_3^5g_4$&$0$&$g_3^2g_5^2$\\
$13$&$g_3g_{10}'$,\  $g_3^2g_{10}'$,\ $g_3g_4g_{13}'$&$0$&$g_3^2g_{11}'$\\
$17$&$g_{17}$,\  $g_5^2g_9$&$0$&$g_9^2$\\
$18$&$g_{18}$&$g_{10}^2$&$g_{10}^2$\\
$21$&$g_{21}+g_{10}g_{11}$,\  $\sq^{12,6,3,1}\io+g_6^3g_7$&$g_{21}+g_{10}g_{11}$&$0$
\end{tabular}
\end{table}

\bigskip
\begin{minipage}{6in}
\begin{fig}\label{figc}

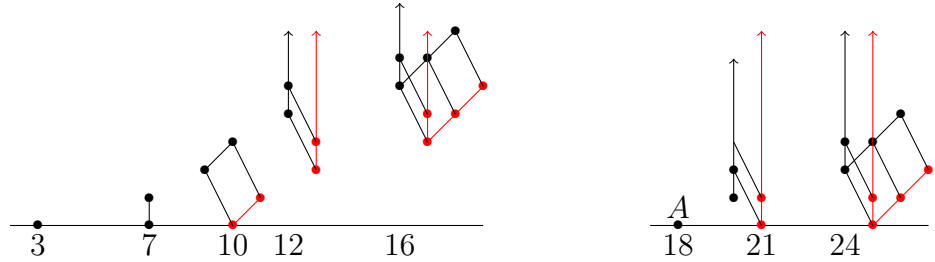
{\bf Ext-chart for $M_3\oplus M_{10}$ (left), and $M_{18}\oplus M_{21}$}

\begin{center}

\begin{\tz}[scale=.37]
\draw (2,0) -- (19,0);
\node at (3,-.7) {$3$};
\node at (3,0) {\sb};
\node at (7,0) {\sb};
\node at (7,1) {\sb};
\draw (7,0) -- (7,1);
\node at (9,2) {\sb};
\node at (10,3) {\sb};
\node at (12,4) {\sb};
\node at (12,5) {\sb};
\draw [->] (12,4) -- (12,7);
\node at (7,-.7) {$7$};
\node at (10,-.7) {$10$};
\node at (12,-.7) {$12$};
\node at (16,-.7) {$16$};
\node at (26,-.7) {$18$};
\node at (29,-.7) {$21$};
\node at (32,-.7) {$24$};
\node at (16,5) {\sb};
\node at (16,6) {\sb};
\node at (17,6) {\sb};
\node at (18,7) {\sb};
\draw [->] (16,5) -- (16,8);
\draw (16,5) -- (18,7);
\draw (9,2) -- (10,3);
\draw (25,0) -- (35,0);
\draw [color=red] (10,0) -- (11,1);
\node [color=red]  at (10,0) {\sb};
\node [color=red] at (11,1) {\sb};
\node [color=red] at (13,2) {\sb};
\node [color=red] at (13,3) {\sb};
\draw [color=red] [->] (13,2) -- (13,7);
\draw [color=red] (17,3) -- (19,5);
\draw [color=red] [->] (17,3) -- (17,7);
\node [color=red] at (17,3) {\sb};
\node [color=red] at (17,4) {\sb};
\node [color=red] at (18,4) {\sb};
\node [color=red] at (19,5) {\sb};
\draw ((10,0) -- (9,2);
\draw (11,1) -- (10,3);
\draw (13,2) -- (12,4);
\draw (13,3) -- (12,5);
\draw (17,3) -- (16,5);
\draw (17,4) -- (16,6);
\draw (18,4) -- (17,6);
\draw (19,5) -- (18,7);
\node at (26,.65) {$A$};
\node at (26,0) {\sb};
\node at (28,1) {\sb};
\node at (28,2) {\sb};
\draw [->] (28,1) -- (28,6);
\draw [->] (32,2) -- (32,7);
\node at (32,2) {\sb};
\node at (33,3) {\sb};
\node at (34,4) {\sb};
\draw (32,2) -- (34,4);
\node at (32,3) {\sb};
\draw [color=red] [->] (29,0) -- (29,7);
\node [color=red] at (29,0) {\sb};
\node [color=red] at (29,1) {\sb};
\draw [color=red] [->] (33,0) -- (33,7);
\node [color=red] at (33,0) {\sb};
\node [color=red] at (33,1) {\sb};
\draw [color=red] (33,0) -- (35,2);
\node [color=red] at (34,1) {\sb};
\node [color=red] at (35,2) {\sb};
\draw (29,0) -- (28,2);
\draw (29,1) -- (28,3);
\draw (33,0) -- (32,2);
\draw (33,1) -- (32,3);
\draw (34,1) -- (33,3);
\draw (35,2) -- (34,4);
\end{\tz}
\end{center}
\end{fig}
\end{minipage}

The differentials follow as before from the fact (\cite{Br} or \cite{Cl}) that $H_{12}(K_3;\Z)\approx\Z/4\approx H_{20}(K_3;\Z)$.

\bigskip
\begin{minipage}{6in}
\begin{fig}\label{figd}

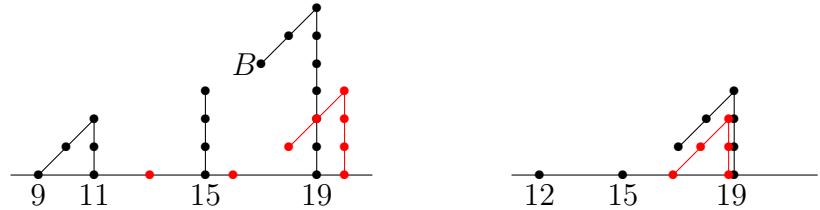
{\bf Ext-charts for $M_{9}\oplus M_{13}$ (left), and $M_{12}\oplus M_{17}$}
\begin{center}

\begin{\tz}[scale=.37]
\draw (8,0) -- (21,0);
\node at (9,-.7) {$9$};
\node at (11,-.7) {$11$};
\node at (15,-.7) {$15$};
\node at (19,-.7) {$19$};
\node at (9,0) {\sb};
\node at (10,1) {\sb};
\node at (11,2) {\sb};
\node at (11,0) {\sb};
\node at (11,1) {\sb};
\node at (16.4,4) {$B$};
\draw (11,0) -- (11,2);
\draw (9,0) -- (11,2);
\node at (15,0) {\sb};
\node at (15,1) {\sb};
\node at (15,2) {\sb};
\node at (15,3) {\sb};
\draw (15,0) -- (15,3);
\node at (19,0) {\sb};
\node at (19,1) {\sb};
\node at (19,2) {\sb};
\node at (19,3) {\sb};
\node at (19,4) {\sb};
\node at (19,5) {\sb};
\node at (19,6) {\sb};
\node at (18,5) {\sb};
\node at (17,4) {\sb};
\draw (19,0) -- (19,6) -- (17,4);
\node [color=red] at (13,0) {\sb};
\node [color=red] at (16,0) {\sb};
\node [color=red] at (20,0) {\sb};
\node [color=red] at (20,1) {\sb};
\node [color=red] at (20,2) {\sb};
\node [color=red] at (20,3) {\sb};
\node [color=red] at (19,2) {\sb};
\node [color=red] at (18,1) {\sb};
\draw [color=red] (20,0) -- (20,3) -- (18,1);
\draw (26,0) -- (37,0);
\node at (27,-.7) {$12$};
\node at (30,-.7) {$15$};
\node at (33.9,-.7) {$19$};
\node at (27,0) {\sb};
\node at (30,0) {\sb};
\node at (34,0) {\sb};
\node at (34,1) {\sb};
\node at (34,2) {\sb};
\node at (34,3) {\sb};
\node at (33,2) {\sb};
\node at (32,1) {\sb};
\draw (34,0) -- (34,3) -- (32,1);
\node [color=red] at (31.8,0) {\sb};
\node [color=red] at (32.8,1) {\sb};
\node [color=red] at (33.8,2) {\sb};
\node [color=red] at (33.8,1) {\sb};
\node [color=red] at (33.8,0) {\sb};
\draw [color=red] (33.8,0) -- (33.8,2)  -- (31.8,0);
\end{\tz}
\end{center}
\end{fig}
\end{minipage}

\bigskip
The only possible differential on the class $A$ in $\ext_{A_1}^{0,18}(M_{18})$ would be to hit the element $B$ in $\ext_{A_1}^{4,21}(M_9)$. However, since $h_1A=0$ but $h_1B\ne0$ such a differential cannot occur. Thus $g_{18}$, which is the desired class $\chi\sq^{15}\io_3$, is a permanent cycle, as claimed. We expect that $d_3$ is nonzero from most of $M_{13}$ to $M_9$, but this is not required for our conclusion.

\def\line{\rule{.6in}{.6pt}}

\end{document}